\documentclass{amsart}
\usepackage{amsmath,amssymb,amsthm}
\usepackage[retainorgcmds]{IEEEtrantools}
\usepackage{hyperref}
\usepackage{verbatim}
\usepackage{enumerate}

\newcommand{\cl}{\text{cl}}
\newcommand{\F}{\mathbb{F}}
\newcommand{\HF}{{}^a\text{HF}}
\newcommand{\N}{\mathbb{N}}
\newcommand{\ld}{{\leq d}}
\newcommand{\p}{\mathfrak{p}}
\newcommand{\R}{\mathbb{R}}
\newcommand{\ord}{\text{ord}}
\newcommand{\In}{\text{in}}

\newcommand{\Spec}{\text{Spec }}
\newcommand{\IE}{IEEEeqnarray*}

\newcommand{\ceiling}[1]{\left\lceil#1\right\rceil}
\newcommand{\floor}[1]{\left\lfloor#1\right\rfloor}
\newcommand{\paren}[1]{\left(#1\right)}

\theoremstyle{plain}
    \newtheorem{theorem}{Theorem}[section]
    \newtheorem{proposition}[theorem]{Proposition}
    
    \newtheorem*{introtheorem}{Theorem}
    \newtheorem{lemma}[theorem]{Lemma}
    \newtheorem{corollary}[theorem]{Corollary}
    
\theoremstyle{definition}

\theoremstyle{remark}

\title[Hilbert Functions in Finite Field Geometry]
    {Hilbert Functions and the Finite Degree Zariski Closure
    in Finite Field Combinatorial Geometry}
\author{Zipei Nie}
\address{Z. Nie: Department of Mathematics, Massachusetts Institute of Technology, Cambridge MA}
\email{zipei@mit.edu} 

\author{Anthony Y. Wang}
\address{A. Y. Wang: Department of Mathematics, Massachusetts Institute of Technology, Cambridge MA}
\email{anthonyw@mit.edu}

\begin{document}
    \begin{abstract}
        The polynomial method has been used recently 
            to obtain many striking
            results in combinatorial geometry.
        In this paper, we use affine Hilbert functions to obtain an
            estimation theorem in finite field geometry.
        The most natural way to state the theorem is via a sort of
            bounded degree Zariski closure operation:
            given a set, we consider all polynomials of some bounded degree
            vanishing on that set, and then the common zeros of these
            polynomials.
        For example, the degree $d$ closure of $d+1$ points on a line
            will contain the whole line, as any polynomial of degree 
            at most $d$ vanishing on the $d+1$ points must vanish on the line.
        Our result is
            a bound on the size of a finite degree closure of a given set.
        Finally, we adapt our use of Hilbert functions to the
            method of multiplicities.
    \end{abstract}
    \maketitle
    \section{Introduction}
        The polynomial method has recently been applied to problems in
            combinatorial geometry.
        The general idea of what is usually termed the polynomial method
            is to find a polynomial vanishing on a set of interest, and
            then to use algebraic methods to recover combinatorial
            information.
        Our paper primarily concerns the use of the polynomial method in the
            geometry over finite fields.
        Dvir's proof of the finite field Kakeya conjecture was a recent
            breakthrough in this area, showing that a Kakeya set in
            $n$ dimensional space over a finite field $\F_q^n$ must have size
            at least $c_nq^n$ for some constant $c_n$ \cite{dvir}.

        In this paper, we wish to broaden the understanding of the application
            of the polynomial method to this and related phenomena.
        We introduce the concept of degree $d$ closure.
        Let $k$ be a field, and $X$ a subset of $k^n$.
        The degree $d$ closure of $Y$ is defined as
        \[ 
            \cl_d(Y) = \{x\in k^n \, | \, \text{every polynomial of degree at most } d
                \text{ vanishing on } Y \text{ vanishes at } x\}.
        \]
        `Taking the degree $d$ closure' seems to be a common operation in
            applications of the polynomial method to combinatorial geometry,
            although not in this terminology.
        For example, saying the degree $d$ closure of $d+1$ points on a line
            contains the line is the same as saying every polynomial of degree $d$
            vanishing on $d+1$ points on a line vanishes on the line.
        The terminology degree $d$ closure was chosen because of its similarity
            to the Zariski closure of a set.
        Indeed, the Zariski closure of $Y$ is 
        \[ 
            \cl(Y) = \{x\in k^n \, | \, \text{every polynomial vanishing on } Y
                \text{ vanishes at } x \},
        \]
        where the only difference is we replace `every polynomial' with
            `every polynomial of degree at most $d$'.
        In our paper we show that for subsets $Y$ of $\F_q^n$, we must have
        \[ 
            \HF(\F_q^n,d)|\cl_d(Y)| \leq q^n |Y|,
        \]
        where $\HF(\F_q^n,d)$ is the affine Hilbert function of the space
            $\F_q^n$ (the dimension of the polynomials of degree at most
                $d$ on $\F_q^n$).
        In particular, if $d<q$, we have
        \[ 
            \binom{d+n}{n}|\cl_d(Y)| \leq q^n|Y|.
        \]
        This inequality gives us a bound on the size of the degree $d$ closure of $Y$
            over finite fields.

        To show the use of our inequality, we note that it will immediately solve
            our summer research project, suggested to us by Larry Guth.
        He asks: given lines $L_1,\cdots,L_c$ in $\F_q^3$, pick subsets
            $\gamma_i\subset L_i$ such that $|\gamma_i|\geq \frac q2$.
        Suppose $X=\cup L_i$, and $Y=\cup \gamma_i$.
        Does there exist a constant $C$ (independent of $q$) such that $|X|\leq C |Y|$?
        
        For this question, we can get such a constant $C$ by noting that the degree 
            $\floor{\frac{q-1}2}<\frac q2$
            closure of $Y$ contains $X$, and then applying the inequality.
        To the extent of our knowledge, the previous best bound was $|X|\leq C\log q |Y|$
            as given in \cite[Lecture 12]{guth}.
        Indeed, the method can be use to prove variants quite easily.
        For example, we can get a fixed constant in $\F_q^n$ for any fixed dimension $n$.
        Moreover, we can prove a similar bound when instead of $X$ being a union of lines,
            it is a union of curves of bounded degree:

        \begin{introtheorem}
        Given curves $\Gamma_1,\cdots,\Gamma_c$ of degree at most $\Lambda$ in $\F_q^n$, 
            pick subsets $\gamma_i\subset \Gamma_i$ such that $|\gamma_i|\geq \frac q 2$.
        Then there is a constant $C$ such that $|\cup \Gamma_i|\leq C |\cup \gamma_i|$.
        \end{introtheorem}

        This follows from the bound by noting that the degree $d$ closure of 
            $\cup \gamma_i$ contains $\cup \Gamma_i$ for $d<\frac{q}{2\Lambda}$.
        Alternatively, we could have assumed instead of $|\gamma_i|\geq \frac q2$ that
            $|\gamma_i|>q^\alpha$ for some $0< \alpha < 1$ to get

        \begin{introtheorem} 
        Given lines $L_1,\cdots,L_c$ in $\F_q^n$, pick subsets
            $\gamma_i\subset \Gamma_i$ such that $|\gamma_i|>q^\alpha$.
        Then $|\cup L_i|\leq n!q^{n(1-\alpha)}|\cup \gamma_i|$.
        \end{introtheorem}

        This follows from the inequality by noting that the degree $q^\alpha$ closure
            of $\cup \gamma_i$ contains $\cup L_i$.
        We can also restate proofs of old bounds in this language.
        For example, in Dvir's proof of the finite field Kakeya conjecture \cite{dvir},
            part of his argument was establishing that the degree $q-1$ closure of
            any Kakeya set is all of $\F_q^n$.
        Indeed, he argues that since for any Kakeya set $K$, there is a line 
            in every direction, any polynomial of degree $\leq q-1$ vanishing on $K$
            must vanish at the hyperplane at infinity.
        But the only affine polynomial vanishing at the plane at infinity is $0$:
            thus the degree $q-1$ closure of $K$ must be $\F_q^n$.
        We can conclude this version of his proof by applying the bound.
        The general argument for this and related proofs is to find the smallest degree
            $d$ of a polynomial vanishing on a set, and then use the fact that for finite
            sets $X$, there exists
            a polynomial of degree $n|X|^{\frac 1n}$ vanishing on it.
        The bound can be viewed as a sort of generalization of this argument.
        If $d$ is the smallest degree of a polynomial vanishing on $X$, then
            the degree $d-1$ closure of $X$ will be the whole space.
        If we are working in the space over a finite field, then our bound will
            give the same bound on the size of $X$ as the original argument.

        To the bound, we first establish that the first $d$
            terms of the affine Hilbert function of $Y$ and $\cl_d(Y)$ agree.
        Then, we use the affine Hilbert function to estimate the size of a set
            of monomials spanning the polynomial functions on $\cl_d(Y)$,
        As we are working over a finite field, 
            the set $\cl_d(Y)$ is finite, so the polynomial functions on it
            are precisely the functions on it and has dimension equal to
            the size of $\cl_d(Y)$.
        Thus, we get a bound on $|\cl_d(Y)|$.

        We also generalize our results to the method of multiplicities
            which has been applied to this area by Saraf and Sudan \cite{saraf-sudan}, and by
            Dvir, Kopparty, Saraf and Sudan \cite{dvir-kopparty-saraf-sudan}.
        This method used the fact that polynomials of bounded degree
            vanishing to high order on a set must vanish to high order 
            on a larger set to get better bounds than was obtained in 
            \cite{dvir} on the size of a Kakeya set.
        Similar to before, we can define an operation $\cl_d^{\ell,m}(Y)$ to be
            the set of points $x$ such that all polynomials of degree at most $d$
            vanishing to order $m$ on $Y$ vanish to order $\ell$ at $x$.
        Whereas before, the degree $d$ closure operation of a set 
            was related to the ring of polynomials on the set,
            this new operation is related to the polynomials on the set up to some
            order of vanishing.
        More precisely, just like we can define the ideal $I(Y)$ to be the
            set of polynomials vanishing at $Y$, we can define the ideal
            $I^m(Y)$ to be the set of polynomials vanishing to order $m$ at
            $Y$.
        Now just as the affine Hilbert function $\HF(Y,d)$ counts the dimension
            of polynomials of degree at most $d$ of the ring $k[x_1,\cdots,x_n]/I(Y)$,
            we can define $\HF^m(Y,d)$ to count the dimension of polynomials
            of degree at most $d$ of the ring $k[x_1,\cdots,x_n]/I^m(Y)$.
        Again, the fact that $X\subset \cl_d^{\ell,m}(Y)$ gives information
            about the modified Hilbert functions: namely that 
            $\HF^\ell(X,d)\leq \HF^m(Y,d)$.
        We can then use the modified Hilbert functions to recover combinatorial
            information about $X$ and $Y$.

        Framed in this manner, we can give a speculative reason why considering
            vanshing to higher multiplicities improves the bounds gotten 
            without considering such higher order vanishings.
        For this, recall the Schwartz-Zippel lemma with multiplicity, which 
            states that if $X$ is a finite subset of a line $L$,
            every polynomial of degree less than $|X|(m-\ell+1)+\ell-1$ vanishing to order
            $m$ on $X$ must vanish to order $\ell$ on $L$, i.e. $L\subset\cl_d^{\ell,m}(X)$
            for such $d$.
        Recall that for proving statements such as the finite field Kakeya and
            the finite field Nikodym (without mulitplicities), we only used lines to show that
            the degree $d$ closure for some set contains some larger set.
        The reason using multiplicity gives better bounds is that $\cl_d^{\ell,m}$ allows us
            to isolate these lines better than $\cl_d$.
        For example, if $Y$ is a set of points, then $\cl_d(Y)$ must contain all lines passing
            through at least $d+1$ of the points of $Y$.
        However, it must also contain all lines passing through at least $d+1$ points of this
            larger set.
        And so on.
        By using $\cl_d^{\ell,m}$ for properly chosen $\ell$, $m$ and $d$, we can reduce the
            amount of baggage that comes along with the `and so on'.
        On the flip other hand, if we know the size of $X$ and that $X\subset |\cl_d^{\ell,m}(Y)|$
            for properly chosen $\ell$, $m$ and $d$,
            we know $X$ was gotten from $Y$ `with minimal baggage', and so our bound
            on the size of $Y$ will be larger.

        As an application of our consideration of multiplicities, we apply it to 
            the problem of Statistical Kakeya for Curves in 
            Dvir, Kopparty, Saraf and Sudan's paper
            \cite{dvir-kopparty-saraf-sudan}.
        Namely, we show 
        \begin{introtheorem}[Statistical Kakeya for Curves]
            Let $X$ and $Y$ be subsets of $\F_q^n$.
            Suppose that for every point $x\in X$, there is a curve
                $C_x$ of degree at most $\Lambda$
                through $x$ which intersects $Y$ in at least $\tau$
                points.
            Then
            \[ 
                |X| \leq \paren{1+\frac{\Lambda(q-1)}{\tau}}^n |Y|.
            \]
        \end{introtheorem}
        The essential difference in our arguments is that they use a polynomial
            vanishing to high order on a set to bound the size of the set, while
            we use information from the modified Hilbert function.
        Although we achieve the same bound as them for the case where the set $X$
            is $\F_q^n$ (which is the case of interest in the applications
            to the Kakeya problem), our bound is strictly better in nearly
            all other cases, namely when the dimension is at least $2$ and
            $X\subsetneq \F_q^n$.

        We expect that using multiplicities will give better bounds
            when we know the set $X\subset \cl_d(Y)$ via an argument using lines directly and
            not some iterated argument.
        In line with this, the statistical Kakeya for curves will improve the constants in
            the some of the bounds we got before from just applying the bound for $\cl_d$,
            such as that for our summer research problem.

        Lastly, we give a bound on $\cl_d^{\ell,m}(Y)$ for $Y\in \F_q^n$.
        We show that
        \[ 
            \HF^\ell(\F_q^n,d) |\cl_d^{\ell,m}(Y)|\leq q^n \binom{m+n-1}{n} |Y|,
        \]
        whose proof follows the same general outline as the proof of
            the bound on $\cl_d(Y)$.

        The general outline of the paper is as follows.

        In section \ref{hilbert-functions}, we give
            preliminaries on affine Hilbert functions.

        In section \ref{monomial-order}, we give
            preliminaries on monomial orders, which allow us to reduce
            combinatorial questions on Hilbert functions of a general
            ideal to that of a monomial ideal.

        In section \ref{FKG}, we prove a bound we need in the sequel via
            the FKG inequality.
        
        In section \ref{zariski-closure}, we define the degree $d$ closure,
            and prove our bound for its size.

        In section \ref{multiplicity}, we adapt our methods
            to higher multiplicities.
        We apply this to the Statistical Kakeya for Curves and to give bounds
            on  $\cl_d^{\ell,m}(Y)$.

        \textbf{Acknowledgements}
        This research grew out of a SPUR (Summer Program
            in Undergraduate Research) project which occurred in the summer
            of 2013.
        SPUR is a program for MIT undergraduates founded by 
            Hartley Rogers and carried out each summer.
        We would like to thank the director of the 2013 SPUR program,
            Slava Gerovitch.
        We would also like to thank Pavel Etingof and Jacob Fox
            for being the faculty advisors for the program and for meeting
            with us every week to discuss our project.
        We would especially like to thank Larry Guth
            for proposing our summer research problem.

        We would like to thank Ben Yang for being our graduate mentor
            over the SPUR program, meeting with us every day to discuss
            our progress on the problem.
        We managed to solve the problem via a much different argument than
            is given in this paper. 

        The majority of the results of this paper were obtained in the fall following our
            summer program.
        We would like to thank Larry Guth, Josh Zahl and Ben Yang for
            looking at a preliminary version of our paper.
        All errors are of course our own.

    \section{Affine Hilbert Functions}\label{hilbert-functions}
        In this section, we review preliminaries and set our notation
            for affine Hilbert functions.
        (Our reference for this material is \cite[Chapter 9 Section 3]{cox}).

        We work over the ring of polynomials 
            $A=k[x_1,\cdots,x_n]$ over a field $k$.
        Let $A_\ld$ denote the polynomials of degree at most $d$.
        For an ideal $I$ of $k[x_1,\cdots,x_n]$, 
            let $I_\ld$ denote the polynomials of degree at most $d$
            in $I$.
        Note that $A_\ld$ and $I_\ld$ are both vector spaces over
            $k$.
        The \textbf{affine Hilbert function} of $I$, denoted by $\HF_I$ 
            is given by
        \[
            \HF_I(d) = \dim A_\ld / I_\ld = \dim A_\ld - \dim I_\ld.
        \]
        As we will only be using the affine Hilbert function throughout
            this paper, we will typically just call it the Hilbert function.
        It is clear that $\HF_I(d)$ is nondecreasing in $d$ and that if
            $I\subset J$ are ideals, then $\HF_I(d)\geq \HF_J(d)$.

        Given a set $X$ of $k^n$, let $I(X)$ denote the ideal of
            polynomial functions vanishing on
            $X$, that is,
        \[ 
            I(X) = \{ P\in k[x_1,...,x_n] \, | \, P(x) = 0 \text{ for all }
                x\in X\}.
        \]
        The \textbf{Hilbert function} of $X$, which we denote by $\HF(X,d)$ 
            is then defined as the Hilbert function of the ideal $I(X)$.
        Again, the Hilbert function $\HF(X,d)$ is nondecreasing in $d$.
        Moreover, if $X\subset Y$, then $I(X)\supset I(Y)$, so
            $\HF(X,d)\leq \HF(Y,d)$.

        Given a polynomial $P\in k[x_1,\cdots,x_n]$ and a finite set
            $Y=\{y_1,\cdots,y_s\}\in k^n$, 
            we can define the evaluation map sending $P$ to its values on
            $Y$, i.e.
        \[ 
            P \mapsto (P(y_1),\cdots,P(y_s)) \in k^{|Y|}.
        \]
        This map is clearly linear.
        The polynomials also surject: the ideals
            $I(\{y_1\}),\cdots, I(\{y_s\})$ are maximal and therefore
            pairwise coprime.
        The Chinese remainder theorem then immediately shows that 
            the polynomial functions surject onto $k[x_1,\cdots,x_n]/I(Y)$.
        Moreover, we see the ideal $I(\{y_1,\cdots,y_s\})
                =I(\{y_1\})\cdots I(\{y_s\})$.

        We conclude:
        \begin{lemma}
            If $Y=\{y_1,\cdots,y_s\}$ is a finite subset of $k^n$ and $d\geq |Y|-1$, then 
            \[ 
                \HF(Y,d) = |Y|.
            \]
        \end{lemma}
        \begin{proof}
            For $j\geq 2$, it is not difficult to find a polynomial vanishing on $y_j$ but
                not on $y_1$.
            (Indeed, some component of $y_j$ and $y_1$ must be different, since otherwise they
                would be the same point.)
            By multiplying these polynomials together and normalizing, we can find a polynomial
                $p_1$ of degree at most $|Y|-1$ which vanishes on $y_2,\cdots,y_s$ but is equal to
                $1$ on $y_1$.
            Similarly, we can find polynomials $p_i$ which vanish on $y_j$ for $j\neq i$ but is
                equal to $1$ on $y_i$.

            Thus, the polynomials of degree at most $|Y|-1$ surject onto $k[x_1,\cdots,x_n]/I(Y)$,
                so $\HF(Y,d)=|Y|$ for $d\geq |Y|-1$.
        \end{proof}

        Unfortunately, this is the best bound on $d$
            that works for all sets $Y$ and all fields $k$.
        To see why, take all $Y$ points to lie on a line.
        However, for finite fields, we can do better:

        \begin{lemma}\label{surjection-lemma}
            We have $I(\F_q^n)=(x_1^q-x_1,\cdots,x_n^q-x_n)$.
            Thus, when working over the space $\F_q^n$, the set of monomials
            \[ 
                S = \{x_1^{\alpha_1} \cdots x_n^{\alpha_n} \, | \,
                    0\leq \alpha_i \leq q-1 \}
            \]
            form a basis for $\F_q[x_1,\cdots,x_n]/I(\F_q^n)$.
        \end{lemma}
        \begin{proof}
            Note that the set of monomials 
            \[
                S=\{x_1^{a_1}\cdots x_n^{a_n} \, | \, 
                    0\leq a_i\leq q-1\}
            \]
                form a spanning set of 
                $\F_q[x_1,\cdots,x_n]/(x_1^q-x_1,\cdots,x_n^q-x_n)$.
            Now the evaluation map on $\F_q^n$ surjects
                $A=\F_q[x_1,\cdots,x_n]$ onto a vector space of dimension
                $q^n$.
            Clearly the kernel $I(\F_q^n)$ 
                contains $(x_1^q-x_1,\cdots,x_n^q-x_n)$.
            Then
            \[ 
                q^n = \dim A/I(\F_q^n)
                    \leq A/
                        (x_1^q-x_1,\cdots,x_n^q-x_n)
                    \leq |S| = q^n.
            \]
            Thus, the inequalities are equalities and
                $I(\F_q^n)=(x_1^q-x_1,\cdots,x_n^q-x_n)$.
            We conclude that
                $S$ forms a basis for $\F_q[x_1,\cdots,x_n]/I(\F_q^n)$.
        \end{proof}
        
        \begin{corollary}\label{hilbert-polynomial-bound}
            If $Y$ is a subset of $\F_q^n$ and $d\geq n(q-1)$, then
            \[ 
                \HF(Y,d) = |Y|.
            \]
        \end{corollary}
        \begin{proof}
            Now the maximum degree of a polynomial in the set of monomials
                $S$ of Lemma \ref{surjection-lemma} is $n(q-1)$, so
                we see that 
                $\F_q[x_1,\cdots,x_n]_{\leq n(q-1)}$ 
                surjects onto the functions on $\F_q^n$.
            In particular, it surjects onto the functions on $Y$.
            Therefore,
            \[ 
                \HF(Y,d) = |Y|,
            \]
            for all $d\geq n(q-1)$.
        \end{proof}

        The evaluation map to get another result on
            Hilbert functions:

        \begin{lemma}\label{union-lemma}
            Let $X_1$, $\cdots$, $X_n$ be subsets of $k^n$ and 
                $X=\bigcup X_i$.
            Then
            \[ 
                \HF(X,d) \leq \sum_{i=1}^n \HF(X_i,d).
            \]
        \end{lemma}
        \begin{proof}
            A degree $\leq d$ polynomial $P$
                on $X$ is
                a degree $\leq d$ polynomial $P_i$ on each of the $X_i$.
            Now the map $P\mapsto (P_1,\cdots,P_n)$ is injective, since
                if a polynomial is $0$ on every $X_i$, then it is $0$ on
                $X$.
            Counting degrees then gives the above bound.

            The more general way to say this is if $I_1$,
                $\cdots$, $I_n$ are ideals of a ring $A$ and
                $I=\bigcap I_i$ is an ideal, then
            \[ 
                \HF_I(d) \leq \sum_{i=1}^n \HF_{I_i}(d).
            \]
            Indeed, there is a projection 
                $A_\ld/I_\ld \to A_\ld/I_{i\,\ld}$, and if the image of 
                a polynomial $P$ is zero for all $I_i$, then the original
                polynomial must have been in $\bigcap I_i$, i.e. it was
                $0$.
            Thus, the given map is injective, and the conclusion follows.
        \end{proof}

    \section{Monomial Orders}\label{monomial-order}
        In this section, we give preliminaries on monomial orders.
        Our references are chapter 2 section 2 and chapter 9 section 3 of
            \cite{cox} and section 15.2 of \cite{eisenbud}.

        First, if $a=(a_1,\cdots,a_n)\in \N^n$,
            then we let $x^a$ denote the monomial $x_1^{a_1}\cdots x_n^{a_n}$.
        We let $|a|=a_1+\cdots+a_n$, so that the degree of $x^a$ is $|a|$.

        A \textbf{monomial order} is a total order on the monomials of the
            ring $k[x_1,\cdots,x_n]$ satisfying the following two conditions:
        \begin{enumerate}[(i)]
            \item $1\leq x^a$ for all $a\in \N^n$.
            \item If $x^a<x^b$, then $x^ax^c<x^bx^c$ for all $c\in \N^n$.
        \end{enumerate}
        We say a monomial order is \textbf{graded} if it refines the
            partial order on monomials given by degree, that is, if
            $\deg(x^a)<\deg(x^b)$, then $x^a<x^b$.

        \begin{lemma}[Well Ordering]
            Given a monomial order, any nonempty subset $S$ of monomials
                has a least element.
        \end{lemma}
        \begin{proof}
            Let $I$ be the monomial ideal generated by the elements of $S$.
            Since $k[x_1,\cdots,x_n]$ is Noetherian, 
                the ideal $I$ is generated by a finite
                number of elements, which we can take to be monomials
                (since $I$ is a monomial ideal, each term of a polynomial in $I$
                is in $I$).
            The smallest generator will then be the smallest element of $S$.
        \end{proof}

        For a nonzero polynomial $P\in k[x_1,\cdots,x_n]$, its initial term is the
            term of $P$ with the largest monomial under some monomial order $>$.
        If $I$ is an ideal, let $\In(I)$ denote the set of initial terms of
            polynomials in $I$ under the order $>$.
        Note that if $x^\alpha$ is in $\In(I)$, then so is every multiple of
            $x^\alpha$.
        
        \begin{theorem}[Macaulay]\label{macaulay-theorem}
            Let $I$ be an ideal of $k[x_1,\cdots,x_n]$, and $>$ be a monomial
                order.
            Let $S$ be the set of monomials which are not in $\In(I)$.
            Then the set $S$ forms a basis for the ring $k[x_1,\cdots,x_n]/I$.
        \end{theorem}
        \begin{proof}
            First, the elements of $S$ are linearly independent.
            For if
            \[ 
                P=\alpha_1x^{a_1} + \cdots + \alpha_nx^{a_n} \in I,
            \]
            then the initial term of $P$ must be an initial term of $I$,
                which is a contradiction if the $x^{a_i}$ are in $S$.

            Next, the elements of $S$ span the quotient.
            To show this, consider the span of the elements of $S$ 
                together with the polynomials in $I$ in the ring 
                $k[x_1,\cdots,x_n]$.
            Suppose the set of polynomials not in this span is nonempty.
            Then there is an polynomial $P$ not in the span with the smallest
                initial term.
            If the initial term of $P$ were in $S$, we can subtract a multiple
                of an element of $S$ to get a polynomial not in the span with
                smaller initial term.
            If the initial term of $P$ were not in $S$, we can subtract a
                polynomial in $I$ to get a polynomial not in
                the span with smaller initial term.
            In either case, we contradict our choice of $P$, so
                the elements of $S$ and the polynomials of $I$
                span $k[x_1,\cdots,x_n]$.
        \end{proof}

        \begin{corollary}\label{macaulay-corollary}
            Let $I$ be an ideal of $k[x_1,\cdots,x_n]$, and $>$ be a
                monomial order which is graded.
            Let $S$ be the set of monomials not in $\In(I)$.
            Then $\HF_I(d)$ is equal to the number of monomials of
                $S$ of degree at most $d$.
        \end{corollary}
        \begin{proof}
            Let $S_\ld$ be the set of monomials in $S$ of degree
                at most $d$.
            Note that for a graded monomial order $<$, the 
                degree of a polynomial and the degree of its
                initial term are the same.
            Using the same proof as in Theorem \ref{macaulay-theorem},
                we can show that $S_\ld$ spans $k[x_1,\cdots,x_n]_\ld/I_\ld$.
        \end{proof}

        There are many examples of graded monomial orders.
        We describe the
            homogeneous lexicographic order.
        In this order, we first order monomials by degree and then
            break ties by the lexicographic order.
        The lexicographic order on $\N^n$ is given by $(a_1,\cdots,a_n)
            <(b_1,\cdots,b_n)$ if $a_j<b_j$ for the first index $j$
            for which $a_i\neq b_i$.
        Then the homogeneous lexicographic order is the order given by
            $x^a< x^b$ if either $|a|<|b|$ or $|a|=|b|$ and $a<b$.
        It is easy to check that this order is a graded monomial order.

    \section{A Bound via the FKG Inequality}\label{FKG}
        Using graded monomial orders, we can reduce properties about the
            Hilbert function of a subset $Y\subset\F_q^n$
            to questions about a set of monomials.
        For simplicity of notation, we will equivalently work with the lattice $\N^n$:
            the monomials of $k[x_1,\cdots,x_n]$ are in one-to-one 
            correspondence with $\N^n$ via the map
            $x_1^{a_1}\cdots x_n^{a_n}\mapsto (a_1,\cdots,a_n)$.
        This map is actually an isomorphism of lattices.
        The order by divisibility on the monomials
            gets taken to the order $\leq$ on $\N^n$
            given by
        \[ 
            (a_1,\cdots,a_n)\leq (b_1,\cdots,b_n)
            \text{ iff } a_i\leq b_i\text{ for all }i.
        \]
        Moreover, the gcd operation gets taken to
            taking the min of each component, while the lcm operation
            gets taken to taking the max of each component. 
        For simplicity of notation, we will identify the set of monomials
            in $k[x_1,\cdots,x_n]$ with $\N^n$ in this section.

        Now given an ideal $I$ of $k[x_1,\cdots,x_n]$, we see that
            $\In(I)$ satisfies the property that if $a\in \In(I)$
            and $b\geq a$, then $b\in \In(I)$, that is, $\In(I)$ is an
            upper set.
        Similarly, if $S$ is the set of monomials not in $\In(I)$,
            then $S$ satisfies the property that if $a\in S$
            and $b\leq a$, then $b\in S$, that is $\In(I)$ is a lower set.

        Consider the situation where $Y$ is a subset of $\F_q^n$.
        Let $S$ denote the set of monomials not in the initial terms $\In(I(Y))$.
        By Macaulay's theorem (Theorem \ref{macaulay-theorem}), the number of elements
            of $S$ span the functions on $Y$, so $|S|=|Y|$.
        we know that $\In(I(Y))$
            contains $\In(I(\F_q^n))$, so by Lemma \ref{surjection-lemma}, all points of $S$ are contained
            in the hypercube $\{0,1,\cdots,q-1\}^n$.
        It is from this set up that we will show
        \begin{theorem} \label{size-bound}
            Let $Y\subset \F_q^n$.
            Then,
            \[ 
                \HF(\F_q^n,d)|Y| \leq \HF(Y,d) q^n.
            \]
        \end{theorem}

        The proof is an easy application of the FKG inequality.
        We will state the inequality here, referring the reader to 
            Alon and Spencer's book \cite[Chapter 6]{alon-spencer}
            for a proof.

        Let $L$ be a finite distributive lattice.
        We say that a nonnegative function $\mu: L\to \R^+$ is 
            log-supermodular if
        \[  
            \mu(x)\mu(y)\leq \mu(x\vee y)\mu(x\wedge y)
        \]
        for all $x,y$ in $L$.
        \begin{theorem}[FKG inequality]
            Let $L$ be a finite distributive lattice.
            Let $\mu, f, g: L \to R^+$ be nonnegative functions on $L$
                such that $\mu$ is log-supermodular and
                $f, g$ are increasing.
            Then
            \[ 
                \paren{\sum_{x\in L} \mu(x)f(x)}
                \paren{\sum_{x\in L} \mu(x)g(x)}
                \leq
                \paren{\sum_{x\in L}\mu(x)}
                \paren{\sum_{x\in L}\mu(x)f(x)g(x)}.
            \]
        \end{theorem}
        \begin{proof}[Proof of Theorem \ref{size-bound}]
            Let $S$ be the set of monomials which are not an initial term of $I(Y)$
                and $T=\{0,1,\cdots,q-1\}^n$ be the set of monomials not an initial term of $I(\F_q^n)$.
            Let $M$ denote the set of monomials of degree at most $d$.
            We let $\mu$ be the indicator function for $T$,
                $f$ be the indicator function for $S$ and
                $g$ be the indicator function for $M$.
            It is easy to check that these functions satisfy the conditions of the FKG
                inequality on $\{0,1,\cdots,q-1\}^n$.
            Applying the inequality (and noting that $S\subset T$) then gives
            \[ 
                |S||M\cap T| \leq |T||S\cap M|.
            \]
            Now $|S|=|Y|$, $|M\cap T|= \HF(\F_q^n,d)$, $|T|=q^n$ and
                $|S\cap M|=\HF(Y,d)$.
            Substituting gives us the desired bound
            \[ 
                \HF(\F_q^n,d)|Y| \leq  \HF(Y,d) q^n.
            \]
        \end{proof}

    \section{Finite Degree Closure}\label{zariski-closure}
        In applications of the polynomial method, one often shows a
            statement of the following form:
            every polynomial of degree at most $d$ vanishing on a set
            $Y$ must also vanish on a set $X$.
        We can view statements like this in a slightly different light:
            given a set $Y$ define the \textbf{degree $d$ closure} of
            $Y$, denoted $\cl_d(Y)$ to be the set of all points $x$ such
            that every polynomial of degree
            at most $d$ vanishing on all of $Y$ vanishes must vanish at $x$
            also.
        Equivalently, we can define $\cl_d(Y)$ to be $V(I(Y)_\ld)$, where
            $V$ of a set of polynomials is the set of points in $k^n$
            which vanish on all those polynomials.
        Then the statement at the beginning of this paragraph is   
            equivalent to the statement that $X\subset \cl_d(Y)$.
        As another example, we can view a large part of Dvir's argument \cite{dvir}
            proving the finite field Kakeya conjecture as an 
            argument establishing that
            the degree $q-1$ closure of a Kakeya set is the whole space
            $\F_q^n$.
        The main result of this section is a bound on the size of
            $\cl_d(Y)$ when working over finite fields.

        The reader will note the similarity of the degree $d$ closure with
            the definition of the Zariski closure.
        The full Zariski closure is too fine for our purposes: the Zariski
            closure of a finite point set is just the finite point set.
        However, if we only allow low degree polynomials, we may be able
            to get additional structures such as lines, 
            which were absent before.

        As a warning, the degree $d$ closure $\cl_d$ is 
            a closure operator, but it does not determine a topology.
        That is, the collection of sets 
        \[
            \mathcal C = \{X \in k^n \,| \, X = \cl_d(X)\},
        \]
        is not closed under finite unions, so it is not the collection of
            closed sets in some topology.

        \begin{proposition}
            The degree $d$ closure is a closure operator.
            That is,
            \begin{enumerate}
                \item $X\subset \cl_d(X)$,
                \item If $X\subset Y$ then $\cl_d(X)\subset \cl_d(Y)$,
                \item $\cl_d(\cl_d(X))=\cl_d(X)$.
            \end{enumerate}
        \end{proposition}
        \begin{proof}
                For (1), every polynomial of degree at most $d$ 
                    vanishing on $X$
                    vanishes on $X$.

                For (2), every degree at most $d$ 
                    polynomial vanishing on $Y$ vanishes on $X$.
                Every degree at most $d$ 
                        polynomial vanishing on $X$ vanishes on
                        $\cl_d(X)$.
                Thus, every degree at most $d$ polynomial vanishing on
                        $Y$ vanishes on $\cl_d(X)$.

                For (3), by (1) and (2), it suffices to show that 
                    $\cl_d(\cl_d(X))\subset \cl_d(X)$.
                Every degree at most $d$ polynomial vanishing on $X$
                    vanishes on $\cl_d(X)$.
                Every degree at most $d$ polynomial vanishing on $\cl_d(X)$
                    vanishes on $\cl_d(\cl_d(X))$.
                Thus, every degree at most $d$ polynomial vanishing on
                    $X$ vanishes on $\cl_d(\cl_d(X))$.
        \end{proof}

        We can use the fact that $X\subset \cl_d(Y)$ to get information
            on the Hilbert functions of $X$ and $Y$.

        \begin{proposition}\label{Hilbert-degree-bound}
            We have the following:
            \begin{enumerate}
                \item If $X\subset \cl_d(Y)$, then $I(X)_\ld
                    \supset I(Y)_\ld$.
                \item If $X\subset \cl_d(Y)$, 
                    then $\HF(X,m)\leq \HF(Y,m)$ for all $m\leq d$.
                \item If $Y\subset X\subset \cl_d(Y)$, 
                    then $I(X)_\ld=I(Y)_\ld$.
                \item If $Y\subset X\subset \cl_d(Y)$, 
                    then $\HF(X,m)=\HF(Y,m)$ for 
                    all $m\leq d$.
            \end{enumerate}
        \end{proposition}
        \begin{proof}
            For (1), this is exactly the assertion that the degree $\leq d$
                polynomials vanishing on $Y$ vanish on $X$.
            (2) then follows from (1).

            For (3), we apply (1) twice.
            First, $X\subset \cl_d(Y)$, 
                so $I(X)_\ld\supset
                I(Y)_\ld.$
            Moreover, we have $Y\subset X\subset\cl_d(X)$, so 
                $I(Y)_\ld\supset I(X)_\ld$.
            Thus, $I(X)_\ld = I(Y)_\ld$.
            (4) then follows from (3).
        \end{proof}

        Saying the above in words may be illuminating.
        We can interpret $\HF(X,d)$ as the
            dimension of the degree at most $d$ polynomials on $X$.
        If $X\subset \cl_d(Y)$, two different degree at most $d$ 
            polynomials $f$, $g$ on $X$ must also
            be different on $Y$, for otherwise $f-g=0$ on $Y$, so $f-g=0$
            on $X$, a contradiction.
        Thus, we have that $\HF(X,d)\leq \HF(Y,d)$.

        One of the things studied in the polynomial method is this: given
            a set $X$, what is the minimal degree of 
            the nonzero polynomials vanishing on $X$.
        When working over infinite fields, this is the same as asking:
            what is the largest degree $d$ such that $\cl_d(X)$ is the
            whole space $k^n$.
        The estimate of the fact that for finite sets $|X|$,
            this degree is at most $n|X|^\frac 1n$ can be restated:

        \begin{proposition}\label{degree-bound}
            Let $X$ be a finite set in $k^n$ where $k$ is an infinite 
                field.
            If $\cl_d(X)=k^n$, then
            \[
                \binom{d+n}{n} \leq |X|.
            \]
            In particular, if $d> n|X|^{\frac 1n}$, 
                then $\cl_d(X)\neq k^n$.
        \end{proposition}
        \begin{proof}
            If $\cl_d(X)=k^n$, then
            \[
                \binom{d+n}{n} = \HF(k^n,d)
                = \HF(X,d) \leq |X|.
            \]
            Using the fact that $\frac{d^n}{n^n}\leq \binom{d+n}{n}$,
                we get $d\leq n |X|^\frac1n$, as desired.
        \end{proof}

        Indeed, we could also do this for $X$ a set of $L$ lines.

        \begin{proposition}
            Let $X$ be the union of $L$ lines in $k^n$
                where $k$ is infinite.
            If $\cl_d(X)=k^n$, then
            \[ 
                \frac 1{d+1}\binom{d+n}{n} \leq L.
            \]
            In particular, if $d>nL^\frac{1}{n-1}$, 
                then $\cl_d(X)\neq k^n$.
        \end{proposition}
        \begin{proof}
            Note that for a line $\ell$, we have that $\HF(\ell,d)=d+1$.
            Thus, using Lemma \ref{union-lemma}, we get that $\HF(X,d)
            \leq (d+1)|L|$.
            The conclusion follows similarly to the previous proposition.
        \end{proof}

        We remark that slight care must be taken in finite fields
            due to the fact that there exist nonzero polynomials which
            vanish on the whole space $\F_q^n$

        We also have propositions of the following form:

        \begin{proposition}
            Let $E_1$, $\cdots$, $E_n$ be subsets of $k$ such that
                $|E_i|>d$ for all $i$.
            Then the degree $d$ closure of $E_1\times \cdots\times E_n$
                is $k^n$.
        \end{proposition}
        \begin{proof}
            This is a weak form of the Combinatorial Nullstellensatz.
            It is an easy consquence of a lemma in Alon's and Tarsi's paper \cite[Lemma 2.1]{alon-tarsi}
                and in Alon's paper on the Nullstellensatz \cite[Lemma 2.1]{alon}.
        \end{proof}
       
        As promised, we will give a bound on $\cl_d(Y)$ for finite fields.
        
        \begin{theorem} \label{finite-field-closure-bound}
            Let $Y$ be a subset of $\F_q^n$.
            Then
            \[ 
                \HF(\F_q^n,d)|\cl_d(Y)| \leq q^n |Y|.
            \]
        \end{theorem}
        
        \begin{proof}
            We apply Theorem \ref{size-bound}, Proposition \ref{Hilbert-degree-bound}, and
                the fact that $\HF(Y,d)\leq |Y|$:
            \[ 
                \HF(\F_q^n,d)|\cl_d(Y)|
                \leq \HF(\cl_d(Y),d)q^n 
                = \HF(Y,d)q^n
                \leq q^n |Y|.
            \]
        \end{proof}

        We remark that we did not use anything special about the field $\F_q$ 
            in proving our bound,
            merely the fact that $I(\F_q^n)=(x_1^q-x_1,\cdots,x_n^q-x_n)$.
        Indeed, our proof carries over entirely to finite subsets $E$ of
            $k^n$ where the complement of $\In(I(E))$ in the monomials of
            $k[x_1,\cdots,x_n]$ is a box.
        For example, for an arbitrary field $k$, let $E_1,\cdots,E_n$ be finite
            subsets of $k$.
        Let $E=E_1\times \cdots \times E_n\subset k^n$.
        For any subset $Y$ of $E$, we can get the bound
        \[ 
            \HF(E,d) |\cl_d(Y)| \leq |E||Y|.
        \]
        Indeed, we need to show that the complement of $\In(I(E))$ in the
            monomials (viewed as a lattice $\N^n$) is the set 
        \[ 
            B=\{0,\cdots,|E_1|-1\}\times \cdots \times \{0,\cdots,|E_n|-1\},
        \]
        which can be done similarly to $\F_q^n$: by noting that all the other monomials
            are clearly an initial term of $I(E)$ and then counting to conclude that
            there are no more.
        The only other change is to apply the FKG inequality with the set $B$ instead
            of the set $\{0,\cdots,q-1\}^n$.
        We leave the details to the reader.

        It is perhaps interesting to study the degree $d$ closure in a more general
            setting.
        In line with our viewing the ring $\F_q[x_1,\cdots,x_n]$ as the polynomial
            functions on $\F_q^n=\Spec \F_q[x_1,\cdots,x_n]/(x_1^q-x_1,\cdots, x_q^n-x_n)$,
            the general framework is probably a ring homomorphism $R=\oplus R_i\to S$
            from a graded ring $R=\oplus R_i$ to a ring $S$.
        Here $R=\oplus R_i$ is viewed as the `polynomial functions' on $\Spec S$.
        We might also want to work with Hilbert functions instead of affine Hilbert
            functions as we have done in this paper:
            to get our theory with Hilbert functions we would precompose with the
            map $k[x_0,x_1,\cdots,x_n]\to k[x_1,\cdots,x_n]$ sending $x_0$ to $1$.

        To illustrate possible uses of our theorem, we 
            give a couple of corollaries:

        \begin{corollary}[Our Summer Research Problem]
            Let $L_1$, $\cdots$, $L_c$ be lines in $\F_q^n$,
                and $X = \bigcup L_i$.
            On each $L_i$, pick a subset $\gamma_i$ such that
                $|\gamma_i|\geq \frac q2$, and let $Y= \bigcup \gamma_i$.
            Then
            \[ 
                |Y| \geq \frac{1}{n!\cdot 2^n} |X|.
            \]
        \end{corollary}
        \begin{proof}
            The degree $\lfloor \frac{q-1}{2}\rfloor<\frac q2$ 
                closure of $Y$
                contains $X$.
            Note that for $d<q$, we have $\HF(\F_q^n,d)=\binom{d+n}{n}\geq \frac{(d+1)^n}{n!}$.
            An application of Theorem \ref{finite-field-closure-bound} then immediately
                gives the result.
        \end{proof}

        \begin{corollary}[Finite field Nikodym Conjecture]
            A Nikodym set $N$ in $\F_q^n$ is a set such that for any point
                $x$, there is a line through $x$ whose intersection with
                $N$ has at least $\frac q2$ elements.
            For any Nikodym set $N$, we must have
            \[ 
                |N| \geq \frac{q^n}{n!\cdot 2^n}.
            \]
        \end{corollary}
        \begin{proof}
            We can take $X=\F_q^n$ and $Y=N$ in the previous corollary.
        \end{proof}

        The theorem can also be used to prove variants:
        \begin{corollary}
            Let $L_1$, $\cdots$, $L_c$ be lines in $\F_q^n$,
                and $X = \bigcup L_i$.
            On each $L_i$, pick a subset $\gamma_i$ such that
                $|\gamma_i|> q^\alpha$
                for some $0<\alpha<1$
                and let $Y= \bigcup \gamma_i$.
            Then
            \[ 
                |X| \leq n!\cdot q^{n(1-\alpha)} |Y|.
            \]
        \end{corollary}
        \begin{proof}
            The set $X$ is contained in the degree $\floor{q^\alpha}$ 
                closure of $Y$.
            Using the fact that for $d<q$, we have $\HF(\F_q^n,d)\geq \frac{(d+1)^n}{n!}$,
                we apply Theorem \ref{finite-field-closure-bound} to get the result.
        \end{proof}

        We remark that the bounds in the corollaries will be improved in the
            next section by the statistical Kakeya for curves.
        Note that in the above corollaries, we've only used the fact
            that the degree $d$ closure of $d+1$ points on a line must contain
            the whole line to show that the degree $d$ closure of some smaller
            set must contain some larger set.
        In this situation, considering vanishing with multiplicities allows
            us to get better bounds.

    \section{Multiplicity}\label{multiplicity}
        In Dvir, Kopparty, Saraf and Sudan's paper \cite{dvir-kopparty-saraf-sudan}, 
            the constant in the finite field
            Kakeya set problem was improved when allowing the polynomials
            to vanish on sets with higher multiplicity.
        We pursue this direction of thought in this section.

        First, we need to recall what vanishing with 
            multiplicity greater than one at a point means.
        A polynomial $P$ vanishes with multiplicity $m$ at a point $p$
            if its Taylor expansion about $p$ has no terms of degree
            less than $m$.
        Equivalently, letting $\p=(x_1-p_1,\cdots,x_n-p_n)$ 
            denote the maximal ideal of functions vanishing at the point 
            $p$, we say $P$ vanishes with multiplicity $m$ at $p$
            if $P\in \p^m$.
        The order of $P$ at the point $p$ is defined to be the largest
            $m$ such that $P\in \p^m$ and is denoted $\ord_p(P)$.
        By convention, when $P=0$, we set $\ord_p(P)=\infty$.
        Note that $\ord_p(PQ)=\ord_p(P)+\ord_p(Q)$.

        The multiplicity of a polynomial at a point can also be phrased
            in terms of Hasse derivatives
        (For a reference on Hasse derivatives, see for 
            example \cite[Section 5.10]{hirschfeld}, or 
            \cite[Section 2]{dvir-kopparty-saraf-sudan}).
        Intuitively, 
            Hasse derivatives are defined so that there is a Taylor
            expansion about every point $p\in k^n$:
        \begin{equation} \label{Taylor-expansion}
            P(x) = \sum_{i_1,\cdots,i_n} D^{i_1,\cdots,i_n}P(p)
                (x_1-p_1)^{i_1}\cdots(x_n-p_n)^{i_n}.
        \end{equation}
        In fields with infinite characteristic, the Hasse derivative
            of order $i_1,\cdots,i_n$ will be $\frac{1}{i_1!\cdots i_n!}$
            the ordinary $i_1,\cdots,i_n$ partial derivative.
        Such a formula does not work in finite characteristic because,
            the number $i_1!\cdots i_n!$ might not be invertible
            (and if it isn't, the $i_1,\cdots,i_n$ partial derivative will
            be zero).
        It is clear that if we have the Taylor
            expansion (\ref{Taylor-expansion}),
            then a polynomial vanishes with multiplicity $m$ 
            at a point $p$ if
            and only if $D^{i_1,\cdots,i_n}P$ vanish at $p$ for all
            $i_1+\cdots+i_n<m$.

        We define
        \[ 
            D^{i_1,\cdots,i_n} P(x_1,\cdots,x_n)
            = [t^{i_1}\cdots t^{i_n}] P(x_1+t_1,\cdots,x_n+t_n).
        \]
        Letting $d$ be the degree of $P$ and $|i| = i_1+\cdots+i_n$, this
            shows that $D^{i_1,\cdots,i_n}P$ is
            a polynomial of degree $d-|i|$.
        For short, if $i=(i_1,\cdots,i_n)$, we let $D^i$ denote
            $D^{i_1,\cdots,i_n}$, and $t^i$ denote 
            $t_1^{i_1}\cdots t_n^{i_n}$.
        From the definition, 
        \begin{equation}\label{Hasse-expansion}
            P(x+t) = 
                \sum_{i\in \N^n} D^{i}P(x)t^i.
        \end{equation}
        Setting $t=x-p$ and $x=p$ in equation (\ref{Hasse-expansion})
            gives the desired equation (\ref{Taylor-expansion}).

        Let $k$ be a field and $X$ a subset of $k^n$.
        We define $I^m(X)$ to be the ideal of all polynomials
                which vanish to order at least $m$ at each point of $X$.
        Note that $I^1(X)$ is just the $I(X)$ which was defined before.
        It is clear that $I^m(X)$ is decreasing in both $m$ and $X$:
            if $m_1\leq m_2$, then $I^{m_1}(X)\supset I^{m_2}(X)$,
            and if $X_1 \subset X_2$, then $I^m(X_1)\supset I^m(X_2)$.

        We define $\HF^m(X,d)$ to be the Hilbert function of $I^m(X)$.
        By the above, we see
            that $\HF^m(X,d)$ is increasing in $m$, $X$ and $d$.

        It is clear from the definitions that if $p$ is a point in $k^n$, then
            $I^m(\{p\})=(I(\{p\}))^m$.
        Let $Y=\{y_1,\cdots,y_s\}$ be a finite subset of $k^n$.
        Consider the evaluation function sending a polynomial $P$ to its values
            and the values of all its Hasse derivatives of order at most $m-1$
            on each of the points of $Y$.
        This image has dimension $\binom{m+n-1}{n}|Y|$: we wish to show it surjects.
        Now $I^m(\{y_1\}),\cdots,I^m(\{y_s\})$ are pairwise coprime
            since $I(\{y_i\})$ is coprime with $I(\{y_j\})$ for $i\neq j$
            (and if $\mathfrak a$ is coprime with $\mathfrak b$ and $\mathfrak c$, 
            then $\mathfrak a$ is coprime with $\mathfrak b\mathfrak c$).
        The Chinese Remainder theorem immediately shows that the polynomial
            functions surject.
        Moreover, it shows that for finite spaces
            $I^m(Y)=I^m(\{y_1\})\cdots I^m(\{y_s\})
                = (I(\{y_1\})\cdots I(\{y_s\}))^m
                =(I(Y))^m$.
        By being careful, we can get a degree bound for when $\HF^m(Y,d)$
            stabilizes:

        \begin{lemma}\label{finite-subset-multiplicity-surjection-lemma}
            Let $Y=\{y_1,\cdots,y_s\}$ be a finite subset of $k^n$.
            Then for all $d\geq 2m|Y|-m-|Y|$,
            \[ 
                \HF^m(Y,d) = \binom{m+n-1}{n}|Y|.
            \]
        \end{lemma}
        \begin{proof}
            For simplicity of notation, let $I^m(y)=I^m(\{y\})$ for points $y$.
            By the above discussion, it suffices to show that the polynomials of degree at
                most $2m|Y|-m-|Y|$ form a spanning set for $\F_q[x_1,\cdots,x_n]/I^m(Y)$.

            We wish to construct polynomials $p_1,\cdots,p_n$ of degree $(s-1)(2m-1)$
                such that $p_i\equiv 1$ mod $I^m(y_i)$
                but $p_i \equiv 0$ mod $I^m(y_j)$ for $j\neq i$.
            We construct $p_1$; the others are constructed similarly.
            For $j \neq 1$, 
                we can find linear polynomials $f\in I(y_1)$ and $g_j\in I(y_j)$ such that
                $f+g_j=1$.
            Then
            \[ 
                1=(f+g_j)^{2m-1}= af^m+bg_j^m,
            \]
            shows that there is a polynomial $h_j=bg_j^m$ of degree $2m-1$ which is in $I^m(y_j)$
            and is congruent to $1$ mod $I^m(y_1)$.
            The polynomial $h_2\cdots h_s$ will then have degree $(s-1)(2m-1)$ and is
                congruent to $1$ mod $I^m(y_1)$ but congruent to $0$ mod $I^m(y_j)$ for
                $j\neq 1$.

            It is not difficult to see that the polynomials of degree $\leq m-1$ span
                the ring
                $\F_q[x_1,\cdots,x_n]/I^m(y_i)$ for all $i$.
            Now by the Chinese remainder theorem, an element $p$ of $\F_q[x_1,\cdots,x_n]/I^m(Y)$
                is determined by the set of its residues in $\F_q[x_1,\cdots,x_n]/I^m(y_i)$
                for all $i$.
            Let $r_1,\cdots,r_n$ be polynomials of degree at most $m-1$ such that
                $p\equiv r_i$ mod $I^m(y_i)$.
            Then we see that $p\equiv r_1p_1+\cdots + r_sp_s$ mod $I^m(Y)$.
            The polynomial $r_1p_1+\cdots+r_sp_s$ has degree at most $(s-1)(2m-1)+m-1=2ms-m-s$,
                thus giving the conclusion.
        \end{proof}

        Again, we can do better in the case of finite fields.

        \begin{lemma}\label{multiplicity-surjection-lemma}
            The ideal $I^m(\F_q^n)$ is generated by the set
            \[ 
                \{(x_1^q-x_1)^{m_1}\cdots (x_n^q-x_n)^{m_n} \, |\, m_1+\cdots+m_n=m\}.
            \]
            In particular, the set
            \[ 
                \{x_1^{m_1}\cdots x_n^{m_n} \, | \, \floor{\frac{m_1}q}+\cdots + \floor{\frac{m_n}q} \leq m-1\}
            \]
            forms a basis for $\F_q[x_1,\cdots,x_n]/I^m(\F_q^n)$.
        \end{lemma}
        \begin{proof}
            The first statement is a consequence of the fact that $I^m(\F_q^n)= (I(\F_q^n))^m$
                since $F_q^n$ is a finite set.
            For the second statement, note that the set of monomials not among the leading terms
                of $I^m(\F_q^n)$ must be contained in the above set.
            This is because it is easy to show all the monomials in the complement are leading terms
                of $I^m(\F_q^n)$.
            But then the size of 
            \[ 
                \{x_1^{m_1}\cdots x_n^{m_n} \,|\, \floor{\frac{m_1}q}+\cdots+\floor{\frac{m_n}q}\leq m-1\}
            \]
            is $\binom{m+n-1}{n}|Y|$, so they must span.
        \end{proof}
        \begin{lemma}\label{multiplicity-hilbert-equality}
            Let $Y$ be a set in $\F_q^n$.
            If $d\geq n(q-1)+(m-1)q$, then
            \[ 
                \HF^m(Y,d) = \binom{m+n-1}{n}|Y|.
            \]
        \end{lemma}
        \begin{proof}
            The claim follows by noting
            the basis given in Lemma \ref{multiplicity-surjection-lemma} for the ring $\F_q[x_1,\cdots,x_n]/I^m(\F_q^n)$ 
            consists only of polynomials of degree at most
            $n(q-1)+(m-1)q$.
        \end{proof}

        If $I$ is a subset of $k[x_1,\cdots,x_n]$, define $V^\ell(I)$ to be
            the set of all points $x\in k^n$ such that every polynomial
            in $I$ vanishes with multiplicity at least $\ell$ at $x$.
        It is clear that $V$ is decreasing in both $\ell$ and $I$.
        Note that $I\subset I^m(V^m(I))$ and $X \subset V^m(I^m(X))$.
        Moreover, we have that $V^m(I^m(V^m(I)))=V^m(I)$ and 
            $I^m(V^m(I^m(X)))=I^m(X)$.

        Similar to before, we define 
            $\cl_d^{\ell,m}(X)=V^\ell(I^m(X)_\ld)$, so that 
            if a polynomial of degree at most $d$ 
            vanishes with multiplicity at least $m$
            at each point of $X$, then it vanishes with multiplicity 
            at least $\ell$ at each point of $\cl_d^{\ell,m}(X)$.
        We see that $\cl_d^{\ell,m}(X)$ is decreasing in $d$ and $\ell$ 
            and increasing in $m$ and $X$.
        Using the above, we get 
            $\cl^{j,\ell}_d(\cl^{\ell,m}_d(X))\subset \cl_d^{j,m}(X)$
            and $\cl^{\ell,\ell}_d(\cl^{\ell,m}_d(X))=\cl^{\ell,m}_d(X)
                = \cl^{\ell,m}_d(\cl^{m,m}_d(X))$.
        
        We finally get to
        \begin{lemma}\label{multiplicity-hilbert-function-bound}
            If $X\subset \cl^{\ell,m}_d(Y)$, then
            \[ 
                \HF^\ell(X,d) \leq \HF^m(Y,d).
            \]
        \end{lemma}
        \begin{proof}
            Indeed,
            \[
                I^\ell(X) \supset I^\ell(\cl^{\ell,m}_d(Y))
                    = I^\ell(V^\ell(I^m(Y)_\ld))
                    \supset I^m(Y)_\ld.
            \]
        \end{proof}

        Note that the operation $\cl^{\ell,m}_d$ may not 
            be a closure operator
            unless $\ell=m$, so it is somewhat of a misnomer.
        (We have used the symbol $\cl^{\ell,m}_d$ because of its obvious
            relation to $\cl_d$.)
        Indeed, we always have that if $X\subset Y$ then
            $\cl_d^{\ell,m}(X)\subset \cl_d^{\ell,m}(Y)$.
        We can show $X\subset\cl_d^{\ell,m}(X)$ when $\ell\leq m$.
        We can show 
            $\cl_d^{\ell,m}(\cl_d^{\ell,m}(X))=\cl_d^{\ell,m}(X)$ when
            $\ell=m$.

        \begin{proposition}[Schwartz-Zippel with multiplicity \cite{dvir-kopparty-saraf-sudan}]
            \label{multiplicity-vanish}
            Let $X$ be a finite subset of $k$.
            If $d<|X|(m-\ell+1)+\ell-1$, then $\cl_d^{\ell,m}(X)=k$.
        \end{proposition}
        \begin{proof}
            For notional purposes, let $i=(i_1,\cdots,i_n)\in \N^n$ and
                $|i|=i_1+\cdots+i_n$.
            Also, let $D^{i}=D^{i_1,\cdots,i_n}$.

            If $P$ is a polynomial of degree $d$ vanishing to order $m$
                at a point in $X$,
                then $D^iP$ is a polynomial of degree $d-|i|$
                vanishing to order $m-|i|$ at that point.
            If $d-|i|<|X|(m-|i|)$, then $D^iP$ must vanish identically.
            When $d-\ell+1 < |X|(m-\ell+1)$, this is true for all
                $|i|\leq \ell-1$, so any polynomial vanishing of order
                $m$ on $X$ must vanish to order $\ell$ on all of $k$,
                as desired.
        \end{proof}

        In Theorem 13 of
            \cite{dvir-kopparty-saraf-sudan},
            Dvir, Kopparty, Saraf and Sudan prove a
            theorem called statistical Kakeya
            for curves.
        Due to Lemma \ref{multiplicity-surjection-lemma}, we can
            improve their bound.

        First, a degree $\Lambda$ curve $C$ in $\F_q^n$ is the set
        \[(
            \{(C_1(\lambda),\cdots,C_n(\lambda)) \,|\, \lambda\in F_q,
                C_1, \cdots, C_n \in \F_q[x]_{\leq \Lambda}\},
        \]
        that is, the values a tuple $(C_1(\lambda),\cdots,C_n(\lambda))$
            takes, where $C_1$, $\cdots$, $C_n$ are polynomials in one
            variable of degree at most $\Lambda$.
        Using Proposition \ref{multiplicity-vanish}, we see that if 
            $X\subset C$ for some degree $\Lambda$ curve and 
            $\Lambda d < |X|(m-\ell+1)+\ell-1$, then
            $C\subset \cl_d^{\ell,m}(X)$.
        We now have
        \begin{theorem}[Statistical Kakeya for Curves]
            \label{statistical-kakeya}
            Let $X$ and $Y$ be subsets of $\F_q^n$.
            Suppose that for every point $x\in X$, there is a curve
                $C_x$ of degree at most $\Lambda$
                through $x$ which intersects $Y$ in at least $\tau$
                points.
            Then
            \[ 
                |X| \leq \paren{1+\frac{\Lambda(q-1)}{\tau}}^n |Y|.
            \]
        \end{theorem}
        We remark that this is the Statistical Kakeya for Curves in
            \cite[Theorem 13]{dvir-kopparty-saraf-sudan}, with $S=X$,
            $K=Y$ and $\tau = \eta q$.
        Our bound is strictly 
            better whenever $X$ is not the whole space $\F_q^n$
            and $n\geq 2$ (also, we drop the condition $\tau>\Lambda$).
        \begin{proof}
            Let $d$, $\ell$ and $m$ be constants to be chosen later.
            When $\Lambda d < \tau(m-\ell+1)+\ell-1$,
                we see that $X\subset \cl_d^{\ell,m}(Y)$.
            If in addition, we have that $d\geq (\ell-1)q+n(q-1)$, then
            \[ 
                \binom{\ell+n-1}{n} |X|
                = \HF^\ell(X,d) \leq \HF^m(Y,d)
                \leq \binom{m+n-1}{n} |Y|.
            \]
            Take $d= (\ell-1)q+n(q-1)$.
            Then the above inequality is true when
            \[ 
                \Lambda((\ell-1)q+n(q-1)) < \tau(m-\ell+1)+\ell-1.
            \]
            Rearranging gives
            \[ 
                \Lambda\ell(q-1) + \Lambda n(q-1)-\tau+1 <
                \tau(m-\ell)+\ell.
            \]
            Set $m = \ceiling{\paren{1+\frac{\Lambda(q-1)}\tau}\ell}$.
            When $\ell > \Lambda n(q-1)-\tau +1$, the above inequality
                is then satisfied.
            Taking the limit as $\ell\to \infty$ in the inequality
            \[ 
                \binom{\ell+n-1}{n}|X| \leq \binom{m+n-1}{n}|Y|,
            \]
            then gives the desired inequality.
        \end{proof}
        \begin{corollary}[Our Summer Research Problem]
            Let $L_1$, $\cdots$, $L_c$ be lines in $\F_q^n$,
                and $X = \bigcup L_i$.
            On each $L_i$, pick a subset $\gamma_i$ such that
                $|\gamma_i|\geq \frac q2$, and let $Y= \bigcup \gamma_i$.
            Then
            \[ 
                |X|\leq \paren{3-\frac{2}{q}}^n |Y|.
            \]
        \end{corollary}
        \begin{proof}
            Just apply the statistical Kakeya for curves with $\Lambda = 1$
                and $\tau = \frac q2$.
        \end{proof}
        \begin{corollary}[Finite field Nikodym Conjecture]
            A Nikodym set $N$ in $\F_q^n$ satisfies
            \[ 
                |N| \geq \frac{q^n}{(3-\frac 2q)^n}.
            \]
        \end{corollary}
        \begin{proof}
            Set $X=\F_q^n$ and $Y=N$ in the previous corollary.
        \end{proof}
        \begin{corollary}
            Let $L_1$, $\cdots$, $L_c$ be lines in $\F_q^n$,
                and $X = \bigcup L_i$.
            On each $L_i$, pick a subset $\gamma_i$ such that
                $|\gamma_i|\geq q^\alpha$
                for some $0<\alpha<1$
                and let $Y= \bigcup \gamma_i$.
            Then
            \[ 
                |X| \leq \paren{1+\frac{q-1}{q^\alpha}}^n|Y|.
            \]
        \end{corollary}
        \begin{proof}
            Apply the statistical Kakeya for curves with $\Lambda = 1$ and
                $\tau=q^\alpha$.
        \end{proof}
        
        We now work towards our bound on $|\cl_d^{\ell,m}(Y)|$.
        \begin{theorem}\label{Hilbert-growth}
            Let $I$ be an ideal of $k[x_1,\cdots,x_n]$.
            Suppose $m_1\geq m_2$.
            Then
            \[ 
                \HF_I(\,m_1)\binom{n+m_2}{n} \leq \HF_I(\,m_2)\binom{n+m_1}{n}.
            \]
        \end{theorem}
        First, we make some definitions.
        We give $\N^n$ the usual partial order where
            $(a_1,\cdots,a_n)\leq (b_1,\cdots,b_n)$ if $a_i\leq b_i$ for all $i$.
        Moreover, for $a=(a_1,\cdots,a_n)\in \N^n$, we let $|a|$ denote the sum
            $a_1+\cdots+a_n$.

        Let $S$ be a subset of $\N^n$.
        We define $S_{\leq d}$ to be the set of points $a\in S$ with $|a|\leq d$.
        Similarly, we let $S_{=d}$ to be the set of points $a\in S$ with $|a|=d$.
        For $v\in \N^n$, we let $S+v$ to be the set of points $a+v$ where $a\in S$.
        Finally, we define
        \[ 
            S^+ = S\cup \bigcup_{i=1}^n (S+e_i)
        \]
        where $e_i=(0,\cdots,0,1,0,\cdots,0)$ is the $i$th unit vector.
        Thus, we think of $S^+$ as the points of $S$ along with those points which are
            one unit above them.
        Recall that $S$ is an upper set if for every $x\in S$ and $y\geq x$, then
            $y\in S$.
        It is clear that $S\subset \N^n$ is an upper set if and only if $S=S^+$.
        We have the following useful lemma:

        \begin{lemma}
            For $S\subset \N^n$,
            \[ 
                |(S^+)_{\leq d+1}| \geq \frac{n+d+1}{d+1} |S_{\leq d}|.
            \]
        \end{lemma}
        \begin{proof}
            We prove the lemma by induction on the sum $n+d$.
            The base cases when $n=1$ or $d=0$ are trivial.

            Suppose $n\geq 2$ and $d\geq 1$ and that we have proved the lemma
                for all smaller sums $n+d$.
            First, we wish to show that $|(S^+)_{=d+1}|\geq \frac{n+d}{d+1}|S_{=d}|$.
            Let $V$ be the projection of $(S^+)_{=d+1}$ onto the first $n-1$ components,
                and $U$ be the projection of $S_{=d}$ onto the first $n-1$ components.
            It is clear that $|(S^+)_{=d+1}|=|V|=|V_{\leq d+1}|$.
            Similarly, we have $|S_{=d}|=|U|=|U_{\leq d}|$.
            Moreover, it is easy to show that $U^+\subset V$.
            Applying the inductive hypothesis, we have
            \[ 
                |(S^+)_{=d+1}| = |V_{\leq d+1}| \geq |(U^+)_{\leq d+1}|
                    \geq \frac{n+d}{d+1} |U_{\leq d}| = \frac{n+d}{d+1}|S_{=d}|,
            \]
            as desired.

            To finish the proof, by the induction hypothesis, we have
                $|(S^+)_{\leq d}|\geq \frac{n+d}{d}|S_{\leq d-1}|$.
            Moreover, since $S^+\supset S$, we also have $|(S^+)_{\leq d}|\geq |S_{\leq d}|$.
            Combining these estimates with that of $|(S^+)_{=d+1}|$ gives
            \begin{\IE}{rCl}
                |(S^+)_{\leq d+1}| &=&
                    |(S^+)_{=d+1}| + |(S^+)_{\leq d}|\\
                    &\geq& \frac{n+d}{d+1}|S_{=d}|+\frac{d}{d+1}\cdot \frac{n+d}{d}|S_{\leq d-1}|
                        + \frac1{d+1} |S_{\leq d}|\\
                    &=& \frac{n+d}{d+1}|S_{\leq d}|+\frac 1{d+1}|S_{\leq d}|\\
                    &=& \frac{n+d+1}{d+1} |S_{\leq d}|,
            \end{\IE}
            as claimed.
        \end{proof}
        \begin{proof}[Proof of Theorem \ref{Hilbert-growth}]
            Let $S=\In(I)$ be the set of monomials which are initial terms of $I$.
            Since $S$ is an upper set, we have $S^+=S$, so $|S_{\leq d+1}|\geq \frac{n+d+1}{d+1}|S_{\leq d}|$.
            By induction, we can show that
            \[ 
                \binom{n+m_2}{n}|S_{\leq m_1}| \geq \binom{n+m_1}{n}|S_{\leq m_1}|
            \]
            for all $m_1\geq m_2$.
            Let $M$ be the set of all monomials, and $T$ be the complement of $S$ in $M$.
            Then $|M_{\leq d}| = |S_{\leq d}|+|T_{\leq d}|$.
            By Corollary \ref{macaulay-corollary}, we see that $\HF_I(d)=|T_{\leq d}|$.
            Since $|M_{\leq d}|=\binom{n+d}{n}$, 
                it is clear $\binom{n+m_2}{n}|M_{\leq m_1}|=\binom{n+m_1}{n}|M_{\leq m_2}|$.
            Subtracting the above inequality from this equality, we conclude
            \[ 
                \HF_I(m_1)\binom{n+m_2}{n} \leq \HF_I(m_2)\binom{n+m_1}{n} ,
            \]
            as desired.
        \end{proof}
        We remark that the proof works for any subset $T\subset \N^n$ which is a lower set:
            if $a\in T$ and $b\leq a$, then $b\in T$.
        For such $T$, we have $\binom{n+m_2}{n}|T_{\leq m_1}|\leq \binom{n+m_1}{n}|T_{\leq m_2}|$
            for all $m_1\geq m_2$.
        We shall use this statement in the following theorem.
        \begin{theorem}\label{multiplicity-set-bound}
            Let $Y\subset \F_q^n$.
            Then
            \[ 
                \HF^m(\F_q^n,d) |Y| \leq \HF^m(Y,d)q^n.
            \]
        \end{theorem}
        \begin{proof}
            Define a function $\varphi:\N^n\to \N^n$ given by
            \[ 
                \varphi(a_1,\cdots,a_n) = \paren{\floor{\frac{a_1}q},\cdots,\floor{\frac{a_n}q}}.
            \]
            For $p\in \N^n$, define $A_p=\varphi^{-1}(p)$.
            It is clear that $|A_p|=q^n$.

            Let $R$ denote the set of monomials not in the initial terms $\In(I^m(F_q^n))$,
                viewed as a subset of $\N^n$.
            Let $M$ denote the set of all monomials.
            From the proof of Lemma \ref{multiplicity-surjection-lemma}, we see that
                $R=\varphi^{-1}(M_{\leq m-1})$.

            Let $S\subset R$ be a lower set.
            We wish to show that $|S||R_{\leq d}|\leq |S_{\leq d}||R|= |S_{\leq d}|\binom{m+n-1}{n}q^n$.
            The idea is to use $\varphi$ to `contract' the situation to one in which we can apply Theorem
                \ref{Hilbert-growth}.

            For any subset $T\subset \N^n$ define
            \[ 
                T(i) = \{ p\in \N^n \,|\, A_p \cap T \geq i\}.
            \]
            Now given $p\in \N^n$, the point $p$ appears in $T(i)$
                for $i=1,\cdots,|A_p\cap T|$.
            Thus,
            \[ 
                \sum_{i=1}^{q^n} |T(i)|
                =\sum_{p} |A_p\cap T|=|T|.
            \]
            Similarly, if $T_1$ and $T_2$ are subsets of $\N^n$, then
                $p$ appears in $T_1(i)\cap T_2(j)$ iff $1\leq i\leq |A_p\cap T_1|$
                and $1\leq j \leq |A_p\cap T_2|$.
            Thus
            \[ 
                \sum_{i,j=1}^{q^n} |T_1(i)\cap T_2(j)| = \sum_p |A_p\cap T_1||A_p\cap T_2|.
            \]
            Now we claim that if $T$ is a lower set, then so is $T(i)$ for each $i$.
            Indeed, if $p_1\geq p_2$, then we can find a bijection $\lambda: A_{p_1}\to A_{p_2}$
                such that $\lambda(x)\leq x$ for all $x\in A_{p_1}$.
            Thus, if $T$ is a lower set, then $x\in T$ implies $\lambda(x)\in T$, so that
                $|T\cap A_{p_1}|\leq|T\cap A_{p_2}|$.
            Thus $T(i)$ is also a lower set.

            Going back to the problem, consider the sets $R_{\leq d}(j)$.
            By the symmetry properties of $R$, we see that $|R_{\leq d}\cap A_p|$ depends only
                on the degree $d$ and $|p|$, the sum of the coordinates of $p$.
            Thus, $R_{\leq d}(j)=M_{\leq f_d(j)}$ for some function $f_d:\N\to \N$.
            Applying the variant of Theorem \ref{Hilbert-growth}, we see that for any lower set $T$,
            \begin{\IE}{rCl}
                |T_{\leq m-1}||R_{\leq d}(j)| = |T_{\leq m-1}|\binom{n+f_d(j)}{n} &\leq& \binom{n+m-1}{n} |T_{\leq f_d(j)}|\\
                    &=& \binom{n+m-1}{n} |T\cap R_{\leq d}(j)|.
            \end{\IE}
            Since $S\subset R$ is a lower set, $S(i)\subset M_{\leq m-1}$ are all lower sets.
            We expand
            \begin{\IE}{rCl}
                |S||R_{\leq d}| &=& \sum_{i,j=1}^{q^n} |S(i)_{\leq m-1}||R_{\leq d}(j)|\\
                    &\leq& \binom{n+m-1}{n}\sum_{i,j=1}^{q^n} |S(i)\cap R_{\leq d}(j)|\\
                    &=& \binom{n+m-1}{n}\sum_p |S\cap A_p||R_{\leq d}\cap A_p|.
            \end{\IE}
            Using the FKG inequality, we can show
            \[ 
                q^n|(S\cap A_p)_{\leq d}| \geq |(A_p)_{\leq d}||S\cap A_p|.
            \]
            Noting that $|R_{\leq d}\cap A_p|$ is either $0$ or $|(A_p)_{\leq d}|$, we see that
            \begin{\IE}{rCl}
                |S||R_{\leq d}|&\leq& q^n \binom{n+m-1}{n}\sum_p |S_{\leq d}\cap A_p| \\
                    &=& q^n \binom{n+m-1}n |S_{\leq d}|=|S_{\leq d}||R|,
            \end{\IE}
            as desired.

            Finally, to prove the theorem, let $S$ be the monomials not among the initial terms
                of $I^m(Y)$.
            We then have
            \[ 
                \HF^m(\F_q^n,d)|S| \leq\HF^m(Y,d) q^n\binom{n+m-1}{n}.
            \]
            By Lemma \ref{multiplicity-surjection-lemma}, we have $|S|=\binom{n+m-1}{n}|Y|$,
                so
            \[ 
                \HF^m(\F_q^n,d)|Y| \leq \HF^m(Y,d)q^n,
            \]
            as claimed.
        \end{proof}
        \begin{theorem}
            Let $Y\subset \F_q^n$.
            Let $X=\cl_d^{\ell,m}(Y)$.
            Then
            \[ 
                \HF^\ell(\F_q^n,d)|X|\leq q^n\binom{m+n-1}{n}|Y|.
            \]
        \end{theorem}
        \begin{proof}
            By Theorem \ref{multiplicity-set-bound}, Lemma \ref{multiplicity-hilbert-function-bound}
                and Lemma \ref{multiplicity-surjection-lemma}, we get
            \[ 
                \HF^\ell(\F_q^n,d)|X| \leq \HF^\ell(X,d)q^n
                    \leq \HF^m(Y,d)q^n
                    \leq q^n\binom{n+m-1}{n}|Y|,
            \]
            which was what we wanted.
        \end{proof}
        We remark that, again, there is nothing special about the field $\F_q$.
        If $E_1,\cdots,E_n$ are finite subsets of $k$ and $E=E_1\times\cdots\times E_n$, 
            then we can follow the same proof as above to show that for $Y\subset E$,
        \[ 
            \HF^\ell(E,d)|\cl_d^{\ell,m}(Y)| \leq |E|\binom{n+m-1}{n}|Y|.
        \]
        The changes are similar to before.
        Let $|E_i|=s_i$.
        Then the set of monomials not an initial term of $I^m(E)$ is 
        \[ 
            \{ (m_1,\cdots,m_n) \,|\, \floor{\frac{m_1}{s_1}}+\cdots+\floor{\frac{m_n}{s_n}}\leq m-1\},
        \]
        as can be shown similarly to before
            by noting that all the others are an inital term of $I^m(E)$ and then counting.
        The only other change is to use the FKG inequality on 
            boxes of size $s_1\times \cdots \times s_n$ instead of size $q\times \cdots \times q$
            boxes in Lemma \ref{multiplicity-set-bound}.

\end{document}